\documentclass[11pt]{amsart}
\usepackage{latexsym,amsfonts,amssymb,amsthm,amsmath,mathrsfs,color,amscd,graphicx,fullpage,parskip,hyperref,bm,bbm,dsfont}

\usepackage{comment}
\includecomment{uncomment}
\usepackage{commath,mathtools,setspace}
\usepackage{paralist}
\usepackage{extarrows}
\usepackage{float}

\excludecomment{comment} 

\newcommand{\s}[1]{{\mathcal #1}}

\newcommand{\bb}[1]{{\mathbb #1}}

\DeclareMathOperator{\argmin}{argmin}

\newtheorem{theorem}{Theorem} 
\newtheorem{corollary}[theorem]{Corollary}

\newtheorem{lemma}[theorem]{Lemma}
\newtheorem{proposition}[theorem]{Proposition}
\newtheorem{problem}[theorem]{Problem}
\newtheorem{definition}[theorem]{Definition}

\newtheorem{remark}[theorem]{Remark}

\numberwithin{equation}{section}
\numberwithin{theorem}{section}

\newcounter{step}
\begin{document}
	
	\title
	{A Model Problem for First Order Mean Field Games with Discrete Initial Data}
	
	\author{P.~Jameson Graber}
	\thanks{This research was supported in part by National Science Foundation thought NSF Grant DMS-2045027.}
	\address{J.~Graber: Baylor University, Department of Mathematics;\\
		Sid Richardson Building\\
		1410 S.~4th Street\\
		Waco, TX 76706\\
		Tel.: +1-254-710- \\
		Fax: +1-254-710-3569 
	}
	\email{Jameson\_Graber@baylor.edu}
	
	\author{Brady Zimmerman}
	\address{B.~Zimmerman: Baylor University,\\
		Sid Richardson Building\\
		1410 S.~4th Street\\
		Waco, TX 76706
	}
	\email{Brady\_Zimmerman1@baylor.edu}
	
	
	\date{\today}   
	
	\begin{abstract}
		In this article, we study a simplified version of a density-dependent first-order mean field game, in which the players face a penalization equal to the population density at their final position.
		We consider the problem of finding an equilibrium when the initial distribution is a discrete measure.
		We show that the problem becomes finite-dimensional: the final piecewise smooth density is completely determined by the weights and positions of the initial measure.
		We establish existence and uniqueness of a solution using classical fixed point theorems.
		Finally, we show that Newton's method provides an effective way to compute the solution.
		Our numerical simulations provide an illustration of how density penalization in a mean field game tends to the smoothen the initial distribution.
	\end{abstract}
	

	\maketitle
	
	
	\section{Introduction}
	
	Mean field games were introduced by Lasry and Lions \cite{lasry06,lasry06a,lasry07} and Caines, Huang, and Malham\'e \cite{huang2006large} in order to model the interactions between large numbers of agents playing a differential game; see for instance \cite{bensoussan2013mean,carmona2017probabilistic,gomes2014mean,achdou2020introduction} for an overview of the field.
	In this article, we are interested in first-order, or deterministic, mean field games.
	Typically, these are modeled using a first-order coupled system of partial differential equations, as follows:
	\begin{equation}
		\label{eq:mfg system}
		\begin{cases}
			-\partial_t u + H(x,\nabla_x u) = F(x,m), \quad x \in \bb{R}^d, t \in (0,T),\\
			\partial_t m - \nabla_x \cdot \del{D_p H(x,\nabla_x u)m} = 0, \quad x \in \bb{R}^d, t \in (0,T),\\
			m(x,0) = m_0(x), \quad x \in \bb{R}^n,\\
			u(x,T) = G\del{x,m(x,T)}, \quad x \in \bb{R}^n.
		\end{cases}
	\end{equation}
	Here $H$ is the Hamiltonian for an optimal control problem, which a representative agent solves for a given population density $m(x,t)$.
	The functions $F$ and $G$ model the cost this player must pay as a function of the density $m$.
	In Nash equilibrium, the density will evolve according to the continuity equation driven by the optimal feedback vector field $-D_p H(x,\nabla_x u)$.
	
	There are now many results on the existence and uniqueness of suitably defined \emph{weak} solutions to System \eqref{eq:mfg system} \cite{cardaliaguet2015weak,graber2014optimal,cardaliaguet2014mean,cardaliaguet2015second} as well as regularity of solutions \cite{prosinski2017global,graber2018sobolev,munoz2022classical}.
	As far as we know, all of these results apply only when the initial distribution $m_0$ is a density, i.e.~it has no singular part with respect to Lebesgue measure.
	It remains an open question how to analyze weak solutions in the case where $m_0$ is replaced with a more general measure, which could have singular parts.
	
	In this article, we contribute to the theory of mean field games by analyzing an example in which the initial measure is entirely discrete, i.e.~it is a weighted sum of Dirac measures.
	We simplify the problem by choosing $F = 0$, $H(x,p) = \frac{1}{2}p^2$, and $G(x,m(x,T)) = m(x,T)$.
	We will also set the space dimension $d = 1$.
	In this case the problem can be viewed as a static, ``one-shot'' game in which a player starting at a point $x$ only has to choose a point $y$ to minimize the cost
	\begin{equation*}
		\frac{\del{x-y}^2}{2T} + m(y,T).
	\end{equation*}
	There is nothing particularly important about the value of $T$, so we will set $T = 1/2$, and we will denote by $f(y)$ the final density $m(y,T)$.
	Hence the cost to a player starting at point $x$ and choosing to move to point $y$ is given by
	\begin{equation}
		\label{eq:cost J}
		J(x,y,f) = (x-y)^2 + f(y),
	\end{equation}
	where $f$ can in principle be any probability density function, i.e.~a non-negative measurable function such that $\int_{-\infty}^{\infty} f(x)\dif x = 1$.
	To avoid dealing with too many measure-theoretic issues, we will impose the restriction that $f$ must be continuous.
	
	The initial distribution will be given by $\displaystyle{m = \sum_{j=1}^n a_j\delta_{x_j}}$ with \mbox{$a_1 + \cdots + a_n = 1, a_j \geq 0$}  and $x_1,\ldots,x_n$ distinct points in $\bb{R}$.
	Thus $m$ is an empirical measure in which $a_j$ is the proportion of players initially located at each $x_j$.
	Without loss of generality, we will assume $x_1 < \cdots < x_n$, i.e.~we put the points in the initial measure in order from left to right.
	
	In this context, a \emph{Nash equilibrium} is a measure $\pi$ on $\mathbb{R} \times \mathbb{R}$ such that:
	\begin{enumerate}
		\item $\pi(A \times \mathbb{R}) = m(A) = \sum_{j=1}^{N}a_j\delta_{x_j}(A), \text{where } \delta_{x_j}(A)=$ 
		$\begin{cases}
			1 \text{ if } x_j \notin A \\
			0 \text{ if } x_j \in A
		\end{cases}$, 
		i.e. first marginal of $\pi$ is $m$, 
		\item $\pi{(\mathbb{R} \times B)} = \int_{B}f(x)\dif{x}$ for a continuous density $f,$ i.e.~the second marginal of $\pi$ is the density $f$, and \\
		\item for $\pi$ - a.e.~$(x,y)$, $y \in \argmin{J(x,\cdot,f)},$ i.e.~$\pi$ couples points $x$ and $y$ such that $y$ is an optimal move starting from $x$.
	\end{enumerate}	
	In other words, given the cost function $J$, a density function $f$ is chosen to possibly be an equilibrium measure. With the knowledge of this function $f$, the players choose the optimal strategy for themselves, resulting in some final distribution. If this final distribution matches $f$, then $f$ is in fact an equilibrium. The idea of a ``Self-Fulfilling Prophecy'' can be helpful to understand this: $f$ is ``prophesied'' to be an equilibrium measure, and, if it turns out to result in an equilibrium, then $f$ in a sense ``fulfilled its prophecy.'' The below diagram demonstrates this phenomenon.
	
	\begin{figure}[h]
		\centering
		\includegraphics[width = 5in]{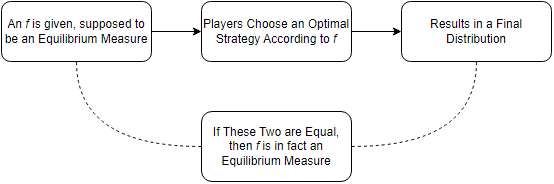 }
		\caption{Mean Field Games Diagram}
		\label{fig:mfg diagram}
	\end{figure}

	In this article, our purpose is two-fold:
	\begin{enumerate}
		\item prove there exists a unique Nash equilibrium, and
		\item compute solutions using a numerical method.
	\end{enumerate}
	As far as we know, this constitutes the first result on existence and uniqueness of solutions to first-order mean field game with singular measures.
	Our approach relies on the fact that we take an entirely discrete measure, which allows us to rewrite the problem in an equivalent finite-dimensional formulation.
	Intuitively, all we need to compute is how each Dirac mass $a_j \delta_{x_j}$ in the initial distribution will ``fan out'' into a density of total mass $a_j$ centered around point $x_j$.
	As explained below, the precise shape of this density is known through a priori considerations.
	This approach makes the problem amenable to classical methods.
	We believe it also provides geometric insight for what is going on more generally in density-penalized mean field games of first order.
	
	\begin{remark}
		The main results of this article were first announced in the second author's undergraduate thesis \cite{zimmerman2024finite}.
	\end{remark}
	
	\section{Reformulation of the equilibrium problem} \label{sec:reformulation}
	
	Let $\pi$ be an equilibrium, and let $f$ be the corresponding final density.
	\begin{definition}
		The support of $\pi$, supp $\pi$, is the set of all $(x,y)$ such that, for all $\epsilon>0$, $\pi(B_{\epsilon}(x,y)) > 0,$ where $B_{\epsilon}(x,y) = \set{(z,w) | \sqrt{(x-z)^2+(y-w)^2}<\epsilon}$, or, equivalently, $\pi(U) > 0$ for every open set containing $(x,y)$.
		
		The support of $f$, $\operatorname{supp} f$, is the closure of the set of all $y$ such that $f(y) > 0$.
	\end{definition}
	\begin{lemma} \label{lem:support}
		If $(x,y) \in \operatorname{supp} \pi$, then $x = x_j$ for some $j$ and $y \in \operatorname{supp} f$.
		Conversely, if $y \in \operatorname{supp} f$, then there exists $x_j$ such that $(x_j,y) \in \operatorname{supp} \pi$.
	\end{lemma}
	\begin{proof}
		If $(x,y) \in \operatorname{supp} \pi$, then for all $\epsilon > 0$ we have $m(B_\epsilon(x)) = \pi(B_\epsilon(x) \times \bb{R}) > 0$ and $\int_{B_\epsilon(y)} f(z)\dif z = \pi(\bb{R} \times B_\epsilon(y)) > 0$.
		It follows that $x = x_j$ for some $j$, and $y \in \operatorname{supp} f$ since for every $\epsilon > 0$ there exists $z \in B_\epsilon(y)$ such that $f(z) > 0$.
		
		Conversely, suppose $y \in \operatorname{supp} f$.
		For every $\epsilon > 0$, $\pi(\bb{R} \times B_\epsilon(y)) = \int_{B_\epsilon(y)} f > 0$.
		Since $\pi(\bb{R} \times B_\epsilon(y)) = \sum_{j=1}^n \pi(\{x_j\} \times B_\epsilon(y))$, for each $\epsilon$ there is a $j$ such that $\pi(\{x_j\} \times B_\epsilon(y)) > 0$.
		We can then find a sequence $\epsilon_k \downarrow 0$ and a fixed $j$ such that $\pi(\{x_j\} \times B_{\epsilon_k}(y)) > 0$ for all $k$, which implies $(x_j,y) \in \operatorname{supp} \pi$.
	\end{proof}
	
	Now if $(x_j,y) \in \operatorname{supp} \pi$, the definition of equilibrium implies $y \in E_j$ where
	\begin{equation*}
		E_j = \argmin\cbr{y : \displaystyle{(x_j-y)^2} + f(y)},
	\end{equation*}
	i.e.~$E_j$ is the set of minimizers $y$ for the cost $J(x_j,y,f)$.
	By Lemma \ref{lem:support}, $\operatorname{supp} f \subset \bigcup_{j=1}^n E_j$; the intuitive meaning is that every player must move into one of the sets $E_j$.
	On the other hand, for $y \in E_j$ we have $f(y) = C_j - {(x_j-y)^2}$, where
	\begin{equation*}
		C_j = \min\cbr{y : \displaystyle{(x_j-y)^2}+ f(y)}.
	\end{equation*}
	Thus, $f$ is completely determined by $C_j$ and $E_j$, and these in turn are coupled by the definition of $E_j$.
	So how do we determine $C_j$ and $E_j$?
	A first clue is the following proposition, whose proof is elementary:
	\begin{proposition} \label{pr:Ej paraboloid}
		$E_j = \cbr{y : C_j - {(x_j-y)^2} \geq \del{C_k - {(x_k-y)^2}}_+ \ \forall k}$, where $x_+ := \max\{x,0\}$.
	\end{proposition}
	
	\begin{proof}
		If $y \in E_j$, then ${(x_j-y)^2} + f(y) = C_j$, and at the same time ${(x_k-y)^2} + f(y) \geq C_k$ for all $k$ by definition of $C_k$.
		Recalling that $f \geq 0$, we see that $C_j - {(x_j-y)^2} = f(y) \geq \del{C_k - {(x_k-y)^2}}_+$ for all $k$.
		
		Conversely, suppose $y \notin E_j$.
		If $y \in E_k$ for some other $k$, then $C_j - {(x_j-y)^2} < f(y) = C_k - {(x_k-y)^2}$.
		If $y \notin \bigcup_1^n E_k$, then $f(y) = 0$, hence $C_j - {(x_j-y)^2} < 0$.
		It follows that $C_j - {(x_j-y)^2} < \del{{C_k} - (x_k-y)^2}_+$ for at least one $k$.
	\end{proof}
	
	\begin{corollary} \label{cor:f max fj}
		Let $f_j(y) = C_j - {(x_j-y)^2}$.
		Then $f(y) = \max\cbr{0,f_1(y),\ldots,f_n(y)}$ and $E_j = \cbr{y : f(y) = f_j(y)}$.
	\end{corollary}
	
	Intuitively, all the players initially concentrated at $x_j$ should spread out according to the density function $f_j$ over the set $E_j$.
	We now make this intuition rigorous.
	
	\begin{proposition} \label{pr:Ej}
		$E_j = \cbr{y : (x_j,y) \in \operatorname{supp} \pi}$, and $\int_{E_j} f = \int_{E_j} f_j = a_j$.
	\end{proposition}
	
	\begin{proof}
		If $(x_j,y) \in \operatorname{supp} \pi$, then by definition of equilibrium, it follows that $y \in E_j$.
		Conversely, suppose $y \in E_j$.
		As explained below in Lemma \ref{lem:Ej is interval}, $E_j$ is an interval, and if $y$ is in the interior of that interval, then $y \notin \cup_{k \neq j}E_k$.
		By Corollary \ref{cor:f max fj} and Lemma \ref{lem:support} we deduce that $(x_j,y) \in \operatorname{supp} \pi$.
		If $y$ is on the boundary of $E_j$, there is a sequence $y_k \to y$ with $y_k$ in the interior of $E_j$, so $(x_j,y_k) \in \operatorname{supp} \pi$ and thus $(x_j,y) \in \operatorname{supp} \pi$.
		
		To prove the remaining statements, start with the following inequality:
		\begin{equation}
			a_j = m(\{x_j\}) = \pi\del{\{x_j\} \times \bb{R}} = \pi\del{\{x_j\} \times E_j} \leq \pi\del{\bb{R} \times E_j} = \int_{E_j} f.
		\end{equation}
		Summing over $j$, we get
		\begin{equation}
			1 = \sum_{j=1}^n a_j \leq \sum_{j=1}^n \int_{E_j} f = \int f = 1.
		\end{equation}
		It follows that $\int_{E_j} f = a_j$, as desired.
	\end{proof}
	With these results, we can now reformulate the problem using only the parameters $C_1,\ldots,C_n$.
	 \begin{definition}\label{def:F(C)}
	 	For a given vector $C = (C_1,\ldots,C_n) \in \bb{R}^n$,
		\begin{itemize}
			\item let $f_j^{(C)}(x) = C_j-(x-x_j)^2$ for $j = 1,\ldots,n$;
			\item let $f^{(C)}(x) = \max\{f_1^{(C)}(x),\ldots,f_n^{(C)}(x),0\}$;
			\item let $E_j^{(C)} = \set{x:f_j^{(C)}(x)=f^{(C)}(x)}$  for $j = 1,\ldots,n$;
			\item let $F_j(C)=\int_{E_j^{(C)}}f^{(C)}(x)\dif{x}$ when $E_j^{(C)}$ is nonempty, and let $F_j(C) = 0$ if $E_j^{(C)}$ is empty,  for $j = 1,\ldots,n$; and
			\item let $F(C)=\del{F_1(C),\ldots,F_n(C)}$.
		\end{itemize}
	\end{definition}
	We remark that, in what follows, $f_j^{(C)}$, $f^{(C)}$, and $E_j^{(C)}$ will often be written more cleanly as $f_j$, $f$, and $E_j$, when it is clear from context that $C$ is fixed.

	To better understand these definitions, consider the following image:
	\begin{figure}[h]\label{Helpful}
		\centering
		\includegraphics[width = 5in]{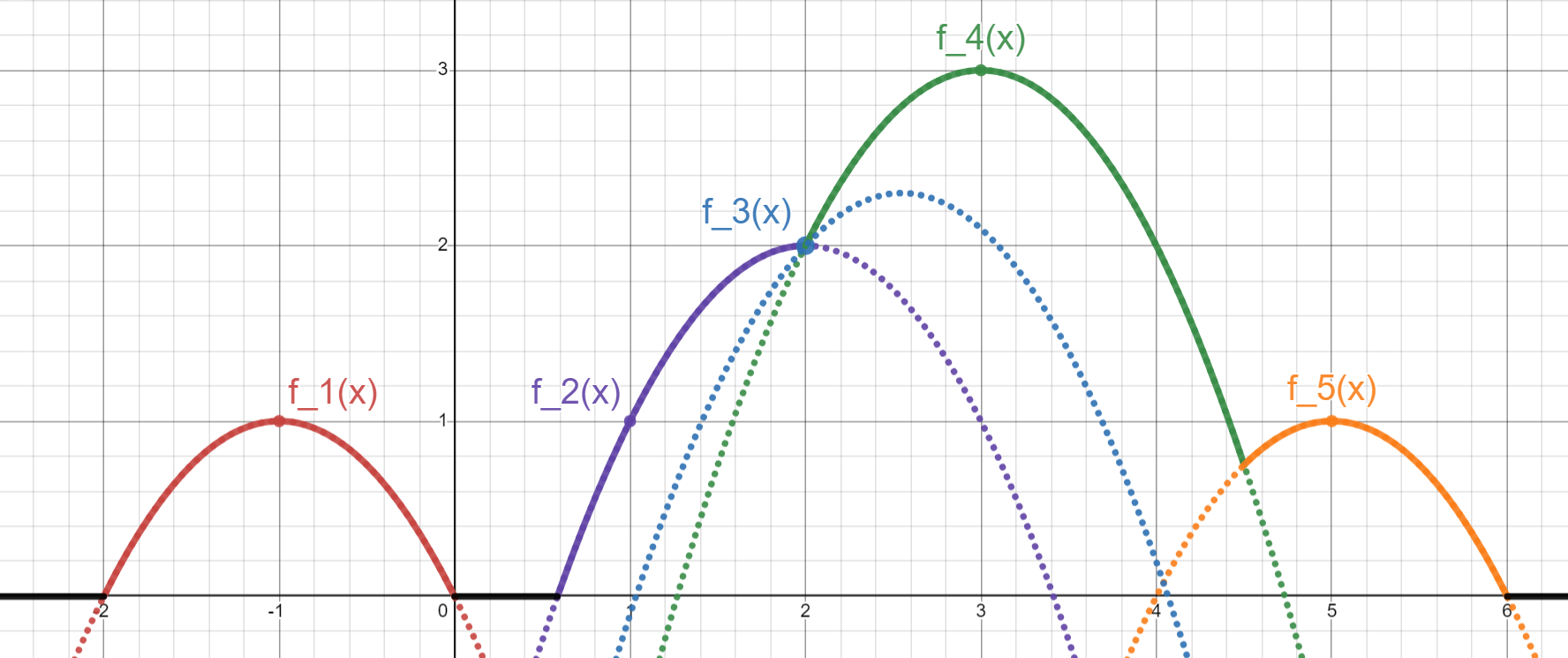}
		\caption{Visualization of the density $f(x)$}
		\label{fig:visualization}
	\end{figure}
	The graph of $f(x)$ is the solid-color part (i.e.~the maximum).
	Each individual curve $f_j(x)$ can be thought of the density after the population sitting at point $x_j$ ``fans out'' and forms a ``bubble" surrounding that point. 
	
	With these definitions, we now state the problem we wish to solve as follows:
	\begin{problem} \label{problem statement}
		Given a vector $a = (a_1,\ldots,a_n)$, whose components are the weights in the discrete measure $m = \sum_{j=1}a_j \delta_{x_j}$, we want to find a vector $C$ such that $F(C)=a$.
	\end{problem}
	
	By the arguments we have just given, if $f(x)$ is an equilibrium density, then $f(x) = f^{(C)}(x)$ where $C$ is the solution to Problem \ref{problem statement}.
	Conversely, suppose $C$ is the solution to Problem \ref{problem statement} and let $f(x) = f^{(C)}(x)$.
	To see that $f(x)$ is an equilibrium density, first define a function $T:\bb{R} \to \bb{R}$ that sends every point in the interior of $E_j = E_j^{(C)}$ to the point $x_j$.
	Notice that $m$ is now the \emph{push-forward} of the density through $T$.
	(We say a measure $\nu$ is the push-forward of a measure $\mu$ through a function $T$ if $\nu(B) = T \sharp \mu(B) := \mu\del{T^{-1}(B)}$; in this case $\nu = m$ and $\dif \mu = f(x)\dif x$.)
	Set $\pi$ to be the push-forward of the density $f$ by the function $y \mapsto (T(y),y)$.
	It follows that $\pi$ is an equilibrium.
	
	We now state our main result:
	\begin{theorem}[Existence and uniqueness] \label{thm:wellposed}
		For each $a$ such that $a_j > 0$ for $j=1,\dots, n$, there exists a unique $C$ such that $F(C)=a$.
	\end{theorem}
	Theorem \ref{thm:wellposed} immediately implies the existence and uniqueness of a Nash equilibrium for any game where the initial distribution is discrete.
	We will prove Theorem \ref{thm:wellposed} in two sections: Section \ref{sec:existence} shows existence, and Section \ref{sec:uniqueness} proves uniqueness.
	These proofs rely on our analysis of $F(C)$ in Section \ref{sec:structure}, in which we establish certain formulas for $F_j(C)$ and its derivatives, prove continuity and coercivity, and the convexity of the effective domain of $F$ (namely, the set of all $C$ for which $F_j(C) > 0$ for all $j$).

	\section{Structure of $F(C)$} \label{sec:structure}
	
	The proof of Theorem \ref{thm:wellposed} is based on various properties of $F(C)$, as given in Definition \ref{def:F(C)}.
	Whenever $C$ is fixed, we will usually suppress the argument $C$ and simply write $f_j(x)$, $f(x)$, and $E_j$ instead of $f_j^{(C)}(x)$, $f^{(C)}(x)$, and $E_j^{(C)}$.
	In Section \ref{sec:Ej}, we study the set $E_j$, namely the set over which $f(x) = f_j(x)$.
	We also study the effective domain of $F(C)$, over which we look for solutions to Problem \ref{problem statement}, and we prove it is convex.
	In Section \ref{sec:formulas} we seek more or less explicit formulas for $F_j(C)$ and its derivatives.
	Section \ref{sec:coercive} shows that $F$ is a \emph{coercive} function on certain domains, while Section \ref{sec:continuity} shows that it is continuous everywhere.
	In reading these technical results, it will be useful to recall that we have ordered the points $x_1,\ldots,x_n$ from left to right, i.e.~$x_1 < \cdots < x_n$ without loss of generality.

	\subsection{Properties of $E_j$} \label{sec:Ej}
	
	We begin with the following claim:
	\begin{lemma} \label{lem:Ej is interval}
		$E_j$ is convex and compact, hence either empty, a singleton, or a closed, bounded interval.\label{Ej}
	\end{lemma}
	\begin{proof}
		First note that $f_j$ is concave and continuous.
		Then observe that for every $i$, $f_j - f_i$ is affine, hence also concave and continuous.
		Indeed,
		\begin{equation} \label{eq:fi-fj}
			f_i(x) - f_j(x) = C_i - C_j + 2(x_i - x_j)\del{x - \frac{x_i + x_j}{2}}.
		\end{equation}
		We can write $E_j$ is the intersection of super-level sets
		\begin{equation}
			E_j = \cap_{i} \{x : f_j(x) - f_i(x) \geq 0\} \bigcap \{x : f_j(x) \geq 0\},
		\end{equation}
		which is therefore convex and closed by the first three observations.
		Moreover, it is compact because $\{x : f_j(x) \geq 0\} = \{x : (x-x_j)^2 \leq C_j\}$ is compact.
		The proof is complete.
	\end{proof}
	In light of this, we introduce the following notation:
	\begin{definition}
		In the case that $E_j$ is and interval, we define $\alpha_j$ and $\beta_j$ as the left and right end points of $E_j$, i.e. $[\alpha_j,\beta_j]=E_j$.
		We denote the length of $E_j$ by $\norm{E_j}=\beta_j-\alpha_j$.
	\end{definition}

	 Before proceeding, we observe that Definition \ref{def:F(C)} leaves open the possibility that $F_j(C) = 0$.
	 Since we are trying to solve $F(C) = a$ where all the components of $a$ are positive, the true domain of interest is given as follows:
	\begin{definition} \label{G}
		$\mathbb{G}=\set{C_j:F(C_j)>0 \text{ for } j = 1, \dots, n}$.
	\end{definition}
	\begin{corollary} \label{cor:Ej interval}
		If $C \in \mathbb{G}$, then $E_j$ is an interval for all $j = 1,\dots,n. $
	\end{corollary}
	\begin{proof}
		First, by definition of $\mathbb{G}$, if $C \in \mathbb{G}$, then $F_j(C)>0$ for all $j=1,\dots,n$. Further, by Definition \ref{def:F(C)}, it is clear that, in order for $F_j(C)$ to be greater than $0$, then $E_j$ cannot be empty or a singleton. Thus, by \ref{Ej}, if $F(C) \in \mathbb{G}$, then $E_j$ must be an interval for all $j=1,\dots,n$.  
	\end{proof}

	The most crucial points of interest will be the intersections between parabolas, namely where $f_i(x) = f_j(x)$, or $C_i - (x-x_i)^2 = C_j - (x-x_j)^2$.
	We define the unique solution by
	\begin{equation}
		\gamma_{ij} =\displaystyle\frac{C_i-C_j}{2(x_j-x_i)}+\frac{x_j+x_i}{2}, i \neq j.\label{gamma}
	\end{equation}	
	By Equation \eqref{eq:fi-fj}, $f_i - f_j$ is increasing when $i > j$ and decreasing when $i < j$.
	It follows that if $x > \gamma_{ij}$ for some $i > j$ or if $x < \gamma_{ij}$ for some $i < j$, then $f_i(x) > f_j(x)$ and thus $x \notin E_j$.
	We deduce that $\beta_j \leq \min\{\gamma_{ij} : i > j\}$ and $\alpha_j \geq \max\{\gamma_{ij} : i < j\}$.
	Conversely, if \mbox{$\max\{\gamma_{ij} : i < j\} \leq x \leq \min\{\gamma_{ij} : i > j\}$}, then we deduce $f_j(x) \geq f_i(x)$ for all $i$.
	It follows that $E_j$ is the intersection of the interval \mbox{$\max\{\gamma_{ij} : i < j\} \leq x \leq \min\{\gamma_{ij} : i > j\}$} with the set where $f_j(x) \geq 0$, i.e.~the interval \mbox{$x_j - \sqrt{C_j} \leq x \leq x_j + \sqrt{C_j}$}.
	We are especially interested in the case when $E_j$ is not empty.
	
	The content of the following lemma is essentially to say that when $C \in \bb{G}$ then Figure \ref{fig:visualization} is exactly the right visualization of the function $f^{(C)}(x)$, in the sense that the points of intersections between parabolas can be ordered from left to right.
	In particular, in order to locate $E^{(C)}_j$, it is sufficient to identify where $f_j^{(C)}(x)$ and its nearest neighbors $f_{j+1}^{(C)}(x)$ and $f_{j-1}^{(C)}(x)$ intersect.
	
	\begin{lemma} \label{lem:alphajformula}
		For $C \in \mathbb{G}$, we have $\min\{\gamma_{ij} : i > j\} = \gamma_{j(j+1)}$ and $\max\{\gamma_{ij} : i < j\} = \gamma_{(j-1)j}$.
		Thus $\alpha_j$ and $\beta_j$ are continuous functions  of $C \in \bb{G}$ given by $\alpha_j = \max\{\gamma_{(j-1)j},x_j-\sqrt{C_j}\}$ and $\beta_j = \min\{\gamma_{j(j+1)},x_j + \sqrt{C_j}\}$, which results in the following formulas:
		\begin{equation} \label{eq:alphajbetaj}
			\alpha_j =
			\begin{cases}
				\displaystyle\frac{C_{j-1}-C_{j}}{2(x_j-x_{j-1})}+\frac{x_j+x_{j-1}}{2} \text{ if } f_j(\alpha_j)>0, \\
				x_j-\sqrt{C_j}, \text{ if } f_j(\alpha_j)=0,
			\end{cases}
			\beta_j = 
			\begin{cases}
				\displaystyle\frac{C_{j}-C_{j+1}}{2(x_{j+1}-x_{j})}+\frac{x_{j+1}+x_{j}}{2}, \text{ if } f_j(\beta_j)>0, \\
				x_j+\sqrt{C_j}, \text{ if } f_j(\beta_j)=0.
			\end{cases}
		\end{equation}	
		Moreover, they satisfy the ordering property $\beta_{j-1} \leq \alpha_j < \beta_j$ for all $j$ (where $\beta_0 := -\infty$ for convenience), and $\beta_{j-1} = \alpha_j$ if and only if $f_j(\gamma_{(j-1)j}) \geq 0$.
	\end{lemma}
	\begin{proof}		
		We claim that $\min\{\gamma_{ij} : i > j\} = \gamma_{j(j+1)}$.
		Suppose to the contrary that there exists some $i > j+1$ such that $\gamma_{ij} < \gamma_{j(j+1)}$.
		We will now show that $E_{j+1}$ is empty.
		Indeed, if $x < \gamma_{j(j+1)}$, then we have $f_j(x) > f_{j+1}(x)$, so $x \notin E_{j+1}$.
		On the other hand, since $\gamma_{j(j+1)} > \gamma_{ij}$, it follows that $f_i(\gamma_{j(j+1)}) > f_j(\gamma_{j(j+1)}) = f_{j+1}(\gamma_{j(j+1)})$.
		Then since $f_i - f_{j+1}$ is increasing, we see that $f_i(x) > f_{j+1}(x)$ for all $x \geq \gamma_{j(j+1)}$, so $E_{j+1}$ is empty.
		This contradicts the assumption that $C \in \bb{G}$, so the claim follows.
		The mirror image of this argument shows that $\max\{\gamma_{ij} : i < j\} = \gamma_{(j-1)j}$.		
		
		We have already established that $E_j$ is the intersection of the interval \mbox{$\max\{\gamma_{ij} : i < j\} \leq x \leq \min\{\gamma_{ij} : i > j\}$} with the set where $f_j(x) \geq 0$, i.e.~the interval \mbox{$x_j - \sqrt{C_j} \leq x \leq x_j + \sqrt{C_j}$}.
		From here it is straightforward to deduce the formulas in Equation \eqref{eq:alphajbetaj}.
		The ordering property $\beta_{j-1} \leq \alpha_j$ follows from $\beta_{j-1} \leq \gamma_{(j-1)j} \leq \alpha_j$, and $\alpha_j < \beta_j$ because $E_j$ is not empty.
		If $\beta_{j-1} = \alpha_j$, then they both must equal $\gamma_{(j-1)j}$, which implies that $f_j(\gamma_{(j-1)j}) \geq 0$.
		Conversely, if $f_j(\gamma_{(j-1)j}) = f_{j-1}(\gamma_{(j-1)j}) \geq 0$, then we deduce that $x_j - \sqrt{C_j} \leq \gamma_{(j-1)j} \leq x_{j-1} + \sqrt{C_{j-1}}$, and this implies $\beta_{j-1} = \alpha_j = \gamma_{(j-1)j}$.
	\end{proof}
	\begin{corollary} \label{cor:length of Ej}
		If $C \in \bb{G}$, then the length of $E_j(C)$ is given by
		\begin{equation}
			\enVert{E_j} = k_j(C) = \min\{\gamma_{j,j+1}(C),x_j + \sqrt{C_j}\} - \max\{\gamma_{j-1,j}(C),x_j - \sqrt{C_j}\}.
		\end{equation}
		In fact, $C \in \bb{G}$ if and only if $k_j(C) > 0$ for all $j$.
		In particular, $\bb{G}$ is a convex set.
	\end{corollary}
	\begin{proof}
		If $C \in \bb{G}$, use Lemma \ref{lem:alphajformula} to see that $\enVert{E_j} = \beta_j - \alpha_j = k_j(C) > 0$.
		If $C \notin \bb{G}$, then at least one $E_j$ must be a singleton or empty, i.e.~$f_j(x)$ is never the maximum in $f(x) = \max\{f_1(x),\ldots,f_n(x),0\}$.
		We deduce that one of the following inequalities must be true:
		\begin{enumerate}
			\item $\gamma_{j(j+1)} < \gamma_{(j-1)j}$
			\item $\gamma_{(j-1)j} > x_j+\sqrt{C_j}$
			\item $\gamma_{j(j+1)} < x_j - \sqrt{C_j}$
		\end{enumerate}
		hence $k_j(C) \leq 0$.
		
		Since $k_j(C)$ is a concave function and $\bb{G} = \{C : k_j(C) > 0 \ \forall j\}$, it follows that $\bb{G}$ is convex.
	\end{proof}

	We make one more simple observation about $\bb{G}$.
	\begin{lemma}
		$\mathbb{G}$ is non-empty.
	\end{lemma}
	\begin{proof}
		To show $\mathbb{G}$ is non-empty, by \ref{G}, we need a $C$ such that $F_j(C)>0$ for all $j=1,\dots,n$. So, let $\epsilon > 0$. Choose $\epsilon$ small enough so that, if $C_j=\epsilon$ for all $j=1,\dots, n$, then $f_j(\alpha_j)=0=f_j(\beta_j)$ for all $j=1,\dots,n$, i.e. no $f_j$ intersects any $f_i, i\neq j$ above the $x$-axis. Then $E_j=[x_j-\sqrt{\epsilon},x_j+\sqrt{\epsilon}]$ for all $j=1,\dots, n$. Thus, we have
		\begin{equation}
			F_j(C)=\frac{4}{3}\epsilon ^{3/2},
		\end{equation}
		which is greater than $0$. Thus, $\mathbb{G}$ is non-empty.
	\end{proof}
	Thus, the set $\mathbb{G}$ is both non-empty and convex, which will be useful in the proof of uniqueness of equilibria.

	\subsection{Computing $F_j(C)$ and its Derivatives} \label{sec:formulas}
	In this section, we want to calculate $F_j(C)$ and all of its derivatives for $C \in \mathbb{G}$. So, we have the following:
	\begin{proposition} \label{pr:Fj}
		$F_j(C)=C_j(\beta_j-\alpha_j)-\frac{1}{3}
		(\beta_j-x_j)^3+\frac{1}{3}(\alpha_j-x_j)^3$ for every $C \in \bb{G}$.
	\end{proposition}
	\begin{proof} We can compute this explicitly:
		\begin{equation}
			F_j(C)=\int_{\alpha_j}^{\beta_j}C_j-(x-x_j)^2 \dif x=C_j(\beta_j-\alpha_j)-\frac{1}{3}
			(\beta_j-x_j)^3+\frac{1}{3}(\alpha_j-x_j)^3.
		\end{equation}
		In particular, taking into account the different possibilities for $\alpha_j$ and $\beta_j$ (see 3.2):
		\begin{equation*}
			\begin{split}
				&F_j(C)=\int_{\alpha_j}^{\beta_j}C_j-(x-x_j)^2 \dif x=C_j(\beta_j-\alpha_j)-\frac{1}{3}
				(\beta_j-x_j)^3+\frac{1}{3}(\alpha_j-x_j)^3\\
				&\text{If } f_j(\alpha_j),f_j(\beta_j)>0,\\
				&F_j(C)=C_j\del{\frac{C_j-C_{j+1}}{2(x_{j+1}-x_j)}+\frac{x_{j+1}+x_j}{2}}-\frac{1}{3}\del{\frac{C_j-C_{j+1}}{2(x_{j+1}-x_j)}+\frac{x_{j+1}+x_j}{2}-x_j}^{3}\\
				&-C_j\del{\frac{C_{j-1}-C_j}{2(x_{j}-x_{j-1})}+\frac{x_{j}+x_{j-1}}{2}}+\frac{1}{3}\del{\frac{C_{j-1}-C_{j}}{2(x_{j}-x_{j-1})}+\frac{x_{j}+x_{j-1}}{2}-x_j}^{3},\\
				&\text{or if } f_j(\alpha_j),f_j(\beta_j)=0,\\
				&F_j(C)=\del{C_j(x_j+\sqrt{C_j})-\frac{1}{3}(x_j+\sqrt{C_j}-x_j)^3}-\del{C_j(x_j-\sqrt{C_j})-\frac{1}{3}(x_j-\sqrt{C_j}-x_j)^3}=\frac{4}{3}C_j^{3/2}.
			\end{split}
		\end{equation*}
		The other two cases, when $f_j(\alpha_j)>0 ,f_j(\beta_j)=0$ and when $f_j(\alpha_j)=0 ,f_j(\beta_j)>0$, are straightforward variations of these two cases.
	\end{proof}
	\begin{remark}
		\label{rem:max Fj}
		These calculations show that, in particular, $0 \leq F_j(C) \leq \frac{4}{3}C_j^{3/2}$ for any $C$.
		This will be useful later.
	\end{remark}
	\begin{proposition}
		$\displaystyle\dpd{F_j(C)}{C_j}({C})=\norm{E_j}+\frac{f_j(\beta_j)}{2(x_{j+1}-x_j)}+\frac{f_j(\alpha_j)}{2(x_j-x_{j-1})}$ for every $C \in \bb{G}$.
	\end{proposition}
	\begin{proof}
		With $F(C)$ now calculated,  we can easily calculate $\dpd{F_j(C)}{C_j}$, taking into account the different possibilities for $\alpha_j$ and $\beta_j$ as above:
		\begin{equation}
			\begin{split}
				&\text{If } f_j(\alpha_j),f_j(\beta_j)>0,\\
				 \dpd{F_j(C)}{C_j}=&\del{\frac{C_j-C_{j+1}}{2(x_{j+1}-x_j)}+\frac{x_{{j+1}}+x_j}{2}}+\del{\frac{C_j}{2(x_{j+1}-x_j)}}-{\frac{1}{2(x_{j+1}-x_j)}}\del{\frac{C_j-C_{j+1}}{2(x_{j+1}-x_j)}+\frac{x_{{j+1}}-x_j}{2}}^2\\
					-&\del{\frac{C_{j-1}-C_j}{2(x_j-x_{j-1})}+\frac{x_{j}+x_{j-1}}{2}}+\del{\frac{C_i}{2(x_j-x_{j-1})}}-{\frac{1}{2(x_j-x_{j-1})}}\del{\frac{C_{j-1}-C_j}{2(x_j-x_{j-1})}-\frac{x_{j}-x_{j-1}}{2}}^2\\
					=&\norm{E_j}+\frac{f_j(\beta_j)}{2(x_{j+1}-x_j)}+\frac{f_j(\alpha_j)}{2(x_j-x_{j-1})},\\
				&\text{or if } f_j(\alpha_j),f_j(\beta_j)=0,\\
				\dpd{F_j(C)}{C_j}= & \ 2C_j.
			\end{split}
		\end{equation}
		Again, the other two cases are straightforward adaptations. In any case, we have
		\begin{equation}
			\dpd{F_j(C)}{C_j}=\norm{E_j}+\frac{f_j(\beta_j)}{2(x_{j+1}-x_j)}+\frac{f_j(\alpha_j)}{2(x_j-x_{j-1})},
		\end{equation}
		which is a continuous function of $C$.
		Notice that this holds for every case; for example, if $f_j(\alpha_j),f_j(\beta_j)=0$, then those components of the partial derivative would be 0, leaving $\norm{E_j}$, which is in fact $2C_j$ in that case.
		In particular, the partial derivative is seen to be continuous everywhere, since as we approach the boundary cases where $f_j(\alpha_j)$ or $f_j(\beta_j)$ goes to zero, the formula still holds in the limit.
	\end{proof}
	\begin{theorem} \label{thm:derivatives}
		For every $C \in \bb{G}$,
		$\dpd{F_j(C)}{C_{j-1}}=\displaystyle\frac{f_j(\alpha_j)}{2(x_j-x_{j-1})}$,$\dpd{F_j(C)}{C_{j+1}}=\displaystyle \frac{f_j(\beta_j)}{2(x_{j+1}-x_{j})}$, and $\dpd{F_j(C)}{C_{i}}= 0$ if $i \notin \{j-1,j,j+1\}$.
	\end{theorem}
	\begin{proof}
		Suppose $\beta_{j-1} < \alpha_j$.
		From Lemma \ref{lem:alphajformula}, we see that $f_j(\gamma_{(j-1)j}) < 0$ and thus $E_j$ does not depend on $C_{j-1}$ for small perturbations of $C$.
		Similarly, if $\beta_j < \alpha_{j+1}$, then $E_j$ does not depend on $C_{j+1}$ for small perturbations of $C$.
		More generally, we deduce that $E_j$ \emph{never} depends on $C_i$ for $C \in \bb{G}$ if $i \notin \{j-1,j,j+1\}$.
		
		So we only consider the case when $f_j(\gamma_{(j-1)j}),f_j(\gamma_{j(j+1)})\geq 0$.
		In this case Lemma \ref{lem:alphajformula} implies that $\alpha_j = \gamma_{(j-1)j}$ and $\beta_j = \gamma_{j(j+1)}$.
		The case $f_j(\gamma_{(j-1)j}) = f_j(\alpha_j) > 0$ is stable under small perturbations of $C$, and we differentiate the formula from Proposition \ref{pr:Fj} to get
		\begin{equation}
			\label{DFj}\dpd{F_j(C)}{C_{j-1}}=-\frac{1}{2(x_j-x_{j-1})}\del{\del{\frac{C_j-C_{j-1}}{2(x_{j-1}-x_j)}+\frac{x_{j-1}-x_j}{2}}^2-C_j}
			=\frac{f_j(\gamma_{(j-1)j})}{2(x_j-x_{j-1})}
			=\frac{f_j(\alpha_j)}{2(x_j-x_{j-1})}.
		\end{equation}
		As $f_j(\gamma_{(j-1)j}) \to 0$, this formula converges to 0, which agrees with the case $f_j(\gamma_{(j-1)j}) < 0$.
		Similarly, the case $f_j(\gamma_{j(j+1)}) = f_j(\beta_j) > 0$ is stable under small perturbations of $C$, and we differentiate the formula from Proposition \ref{pr:Fj} to get
		\begin{equation}
			\dpd{F_j(C)}{C_{j+1}}={\frac{1}{2(x_j-x_{j+1})}\del{\del{\frac{C_{j+1}-C_j}{2(x_j-x_{j+1})}-\frac{x_j-x_{j+1}}{2}}^2-C_j}}
			=\frac{f_j(\gamma_{j(j+1)})}{2(x_{j+1}-x_j)}
			=\frac{f_j(\beta_{j})}{2(x_{j+1}-x_j)}.
		\end{equation}
		As $f_j(\gamma_{j(j+1)} \to 0$, this formula converges to 0, which agrees with the case $f_j(\gamma_{j(j+1)}) < 0$.
		This establishes the existence and continuity of the partial derivatives.
	\end{proof}
	We notice the following, deduced from the formulas in Lemma \ref{lem:alphajformula} and the result of Theorem \ref{thm:derivatives}:
	\begin{remark} \label{rem:potential}
		\begin{equation*}
			\dpd{F_i(C)}{C_j}=\dpd{F_j(C)}{C_i}
		\end{equation*}
		for every $C \in \bb{G}$ and for every $i,j \in \{1,\ldots,n\}$.
		
		This implies that $F$ is in fact a gradient, i.e.~there exists a smooth function $\Phi:\bb{G} \to \bb{R}$ such that $F(C) = \nabla \Phi(C)$ for all $C \in \bb{G}$.
		This is a remarkable coincidence, as it is already well-known that the mean field game \eqref{eq:mfg system} is a \emph{potential game}, meaning that there exists a potential whose critical points are Nash equilibria.
		However, it must be emphasized that usually this potential is defined on the space of (paths defined on) probability measures, and its derivative is supposed to be the cost paid to individual players.
		See e.g.~\cite{benamou2017variational,cardaliaguet2017learning,graber2024remarks}, \cite[Section 2.6]{lasry07}, \cite{carmona2017probabilistic} and references therein.
		Here the context is quite different, as we have significantly transformed the problem, so the ``potential'' in this case does not have precisely the same meaning as in these references.
	\end{remark}

	\subsection{Coercivity} \label{sec:coercive}
	We say that the vector field $F(C)$ is coercive provided that $\displaystyle\lim\limits_{\norm{C} \to \infty}\frac{F(C)\cdot C}{\norm{C}}=\infty$, where $\enVert{C} = \del{C_1^2 + \cdots + C_n^2}^{1/2}$ is the norm of the vector $C$.
	This is the usual condition for a vector field to be surjective, as in the Minty-Browder Theorem \cite{browder-minty}.
	There is no chance for $F(C)$ to be coercive on all of $\bb{R}^n$, since $F_j(C) = 0$ whenever $C_j \leq 0$.
	However, this will not matter, is we are always aiming at restricting to vectors $C$ with non-negative components.
	Instead, we will use the following substitute for coercivity to prove existence of solutions.
	\begin{lemma} \label{lem:coercive}
		There exist constants $\delta \in (0,1)$ and $M > 0$ such that if $C$ is any vector in $\bb{R}^n$ with $C_j = \max \{C_1,...,C_n\} \geq M$, then $F_j(C) \geq \delta C_j$.
	\end{lemma}
	\begin{proof}
		Choose a $C_j$ such that $C_j=\max\{C_1,...,C_n\} \geq M$. We now seek to estimate $F_j({C})$. Recall that, from the discussion before Lemma \ref{lem:alphajformula}, $E_j$ is the intersection of the interval\\ \mbox{$\max\{\gamma_{ij} : i < j\} \leq x \leq \min\{\gamma_{ij} : i > j\}$} with the interval \mbox{$x_j - \sqrt{C_j} \leq x \leq x_j + \sqrt{C_j}$}. 
		By the maximality of $C_j$,
		\begin{equation*}
			\begin{split}
				\max\{\gamma_{ij} : i < j\} &\leq \max\cbr{\frac{x_j+x_i}{2}, i < j} = \frac{x_j + x_{j-1}}{2},
			\\
			\min\{\gamma_{ij} : i > j\} &\leq \min\cbr{\frac{x_j+x_i}{2}, i > j} = \frac{x_j + x_{j+1}}{2}.
			\end{split}
		\end{equation*}		
		We can assume that
		\begin{equation*}
			M \geq \max\cbr{\del{\frac{x_{i+1}-x_i}{2}}^2 : i = 1,\ldots,n-1}.
		\end{equation*}
		Then it follows that
		\begin{equation*}
			x_j - \sqrt{C_j} \leq \frac{x_j + x_{j-1}}{2},
			\quad
			x_j + \sqrt{C_j} \geq \frac{x_j + x_{j+1}}{2}.
		\end{equation*}
		We can now conclude that $E_j = [\alpha_j,\beta_j]$ with
		\begin{equation*}
			\alpha_j \leq \frac{x_{j-1}+x_j}{2}, \quad \beta_j \geq \frac{x_j+x_{j+1}}{2}.
		\end{equation*}
		This gives a lower bound on how close $\alpha_j$ and $\beta_j$ can get to $x_j$.
		Thus,
		 \begin{equation*}
			\begin{split}
				F_j({C})\geq \int_{\frac{x_{j-1}+x_j}{2}}^{\frac{x_j+x_{j+1}}{2}}C_j-(x-x_j)^2 \dif x = C_j\frac{x_{j+1}-x_{j-1}}{2}-\frac{1}{3}\del{\del{\frac{x_{j+1}-x_j}{2}}^3-\del{\frac{x_{j}-x_{j-1}}{2}}^3}.
			\end{split}
		\end{equation*}
		Define
		\begin{equation*}
			\bar{\delta} := \frac{\min_j\set{x_{j+1}-x_{j-1}}}{2}, \quad \bar{M} := \frac{1}{3}\max_j \del{\del{\frac{x_{j+1}-x_j}{2}}^3-\del{\frac{x_{j}-x_{j-1}}{2}}^3}_+,
		\end{equation*}
		where we recall $x_+ := \max\{x,0\}$.
		Then
		\begin{equation*}
			F_j(C) \geq \bar{\delta}C_j - \bar{M}.
		\end{equation*}
		Pick any $0 < \delta < \min\{\bar{\delta},1\}$ and assume $M \geq (\bar{\delta}-\delta)^{-1}\bar{M}$.
		If $C_j \geq M$, we get $F_j(C) \geq \delta C_j$, as desired.
	\end{proof}

	We remark that Lemma \ref{lem:coercive} can be used to prove that $F$ is coercive on $\bb{R}^n_{\geq 0}$, which is defined to be the set of all $(x_1,\ldots,x_n)$ such that $x_j \geq 0$ for every $j$.
	The reason is that if $C \in \bb{R}^n_{\geq 0}$ and $C_j = \max\{C_1,\ldots,C_n\}$, then
	\begin{equation*}
		F(C) \cdot C \geq F_j(C)C_j.
	\end{equation*}
	For the remaining details, let us first define the norm
	\begin{equation*}
		\enVert{x}_\infty = \max\cbr{\abs{x_1},\ldots,\abs{x_n}}
	\end{equation*}
	and note that $\enVert{x}_\infty = \max\{x_1,\ldots,x_n\}$ for $x \in \bb{R}^n_{\geq 0}$.
	Recall that
	\begin{equation*}
		\enVert{x}_\infty \leq \enVert{x} \leq \sqrt{n}\enVert{x}_\infty.
	\end{equation*}
	Let $M,\delta$ be as in Lemma \ref{lem:coercive}.
	If $C \in \bb{R}^n_{\geq0}$ and $\enVert{C} \geq M\sqrt{n}$, then $\enVert{C}_\infty \geq M$ and so
	\begin{equation*}
		F(C) \cdot C \geq F_j(C)C_j \geq \delta C_j^2 = \delta \enVert{C}_\infty^2
		\geq \frac{\delta}{\sqrt{n}}\enVert{C}^2.
	\end{equation*}
	Thus $F$ is coercive on $\bb{R}^n_{\geq0}$.
	However, we also note that only Lemma \ref{lem:coercive}, and not coercivity per se, will be used directly to prove the existence of solutions.
	
	\subsection{Continuity} \label{sec:continuity}
	With coercivity proved, we move to prove that $F(C)$ is continuous. Before doing so, we need to explore what happens in $E_j = \varnothing$.
	\begin{lemma} \label{lem:CIopen}
		Let $I \subseteq \cbr{1,\dots,n}$ and $\mathcal{C}_I$ be the set of all $C$ such that $E_j(C) = \varnothing$ for all $j \in I$. Then $\mathcal{C}_I$ is open, i.e. for every $C \in \mathcal{C}_I$, there exists $\epsilon > 0$ such that, if $\enVert{\tilde{C} - C} < \epsilon$, then $\tilde{C} \in \mathcal{C}_I$.
	\end{lemma}
	\begin{proof}
		Let $C \in \s{C}_I$.
		Define $\tilde f^{(C)}(x) = \max\{f_j^{(C)}(x) : j \in I\}$.
		Then $f^{(C)}(x) > \tilde f^{(C)}(x)$ for all $x$, because all of the sets $E_j(C)$, $j \in I$, are empty.
		Since all of the $f_j^{(C)}(x)$ are parabolas that go to $-\infty$ as $\abs{x} \to \infty$, there exists $M > 0$ such that if $\abs{x} > M$, then $f^{(C)}(x) - \tilde f^{(C)}(x) \geq 1$.
		Now $f^{(C)}(x) - \tilde f^{(C)}(x)$ is a continuous and positive function, hence it has a minimum $\delta > 0$ on the compact set $[-M,M]$.
		Without loss of generality, we can assume $\delta \leq 1$; in particular, we have $f^{(C)}(x) - \tilde f^{(C)}(x) \geq \delta$ for every $x$.
		Now assume $\abs{\tilde C_j - C_j} < \delta/2$ for $j = 1,2,\ldots,n$.
		Then $\abs{f_j^{(\tilde C)}(x) - f_j^{(C)}(x)} < \delta/2$ for every $j$ and every $x$.
		It follows that $\abs{f^{(\tilde C)}(x) - f^{(C)}(x)} < \delta/2$ and $\abs{\tilde f^{(\tilde C)}(x) - \tilde f^{(C)}(x)} < \delta/2$, so 
		\begin{equation}
			\begin{split}
				f^{(\tilde C)}(x) - \tilde f^{(\tilde C)}(x)
				&= f^{(\tilde C)}(x) - f^{(C)}(x) + f^{(C)}(x) - \tilde f^{(C)}(x) + \tilde f^{(C)}(x) - \tilde f^{(\tilde C)}(x)\\
				&> -\delta/2 + \delta - \delta/2 = 0.
			\end{split}
		\end{equation}
		It follows that all the sets $E_j(\tilde C)$, $j \in I$, are empty, and so $\tilde C$ must also be in $\s{C}_I$. Thus, $\mathcal{C}_I$ is open.
	\end{proof}
	With this in mind, we can now prove continuity:
	\begin{theorem}
		$F(C)$ is continuous on $\bb{R}^n$.
	\end{theorem}
	\begin{proof}
		It is sufficient to prove that each component function $F_j(C)$ is continuous.
		First, let us point out that we can remove certain components from consideration if the corresponding sets $E_j(C)$ are empty.		
		Let $I$ be the set of all indices $j$ such that $E_j(C) = \emptyset$, so that $C \in \mathcal{C}_I$, the set defined in Lemma \ref{lem:CIopen}. Take a sequence of vectors $C^{(k)} \rightarrow C$. Then, for $k$ sufficiently large, $C^{(k)} \in \mathcal{C}_I$ also. So $F_j(C^{(k)}) = F_j(C) = 0$ for every $j \in I$. In light of this, we can remove the indices in $I$ and relabel the sequence $\set{1,\dots,n} \setminus I$, calling it $\set{1,\dots,m}$. Then $f(x)$ will still be the $\max\{f_1(x),f_2(x),\dots,f_m(x),0\}$, as those indices which may have been removed will not affect the maximum. Thus, we can assume (without loss of generality) for the remainder of this proof that $E_j(C) \neq \varnothing$ for every $j$, and thus we will be able to appeal to previously derived explicit formulas for the component functions $F_j(C)$.
		
		Let $j$ be fixed but arbitrary.
		Since $E_j(C)$ is not empty, by Lemma \ref{lem:Ej is interval} it is either an interval or a single point.
		Suppose that $E_j(C)$ is an interval.
		Choose a sequence of vectors $C^{(k)} \rightarrow C$. By Proposition \ref{pr:Fj}, $F_j(C)$ is simply a polynomial in terms of $C$, and thus $F_j(C^{(k)}) \rightarrow F_j(C)$.
		
		Lastly, suppose $E_j(C)$ is a single point. Then, by the contra-positive of Corollary \ref{cor:Ej interval}, $F_j(C)=0$. Again, we choose a sequence of vectors $C^{(k)} \rightarrow C$. So, we need to show $F_j(C^{(k)}) \rightarrow 0$. 
		If $F_j(C^{(k)}) > 0$ then $E_j(C^{(k)})$ is an interval whose length is given by $k_j(C^{(k)})$ defined in Corollary \ref{cor:length of Ej}.
		Further, since $k_j(C)$ is a continuous function of $C$ and $E_j(C)$ is a single point, then $k_j(C^{(k)})\rightarrow 0$ as $k \rightarrow \infty$. Therefore, whenever $E_j(C^{(k)})$ is an interval, its length shrinks to zero as $k \rightarrow \infty$, and since in any other case we have $F_j(C^{(k)}) = 0$, we deduce that $F_j(C^{(k)}) \rightarrow 0$ as $k \to \infty$.
	\end{proof}
	We have now sufficiently explored various properties of $f(x)$ and $F(C)$, so we may now move onto our first goal:
	\section{Proving Existence of a Solution} \label{sec:existence}
	In this section, we seek to show that a solution $C$ to the equation $F(C)=a$ exists for all $a$ with positive components. Before we begin the proof, we state Brouwer's Fixed Point Theorem: 
	\begin{lemma}\label{Brouwer}
		Every continuous function from a nonempty convex compact subset $K$ of a Euclidean space to $K$ itself has a fixed point.
	\end{lemma}
	See e.g.~\cite{florenzano2003general}. We now begin the main result:
	\begin{theorem}\label{MAIN}
		Let $a = (a_1,\ldots,a_n)$ with $a_j > 0$ for each $j$.
		Then there exists a solution $C$ to the equation $F(C)=a$.
	\end{theorem}
	
	\begin{proof}
		Let $a$ be given.  Define $L:\mathbb{R}^n_{\geq 0} \rightarrow \mathbb{R}^n_{\geq 0}$ as $L_i(C) := \del{C_i + a_i - F_i(C)}_+$ (recall $x_+ := \max\{x,0\}$).
		We will use the norm
		\begin{equation*}
			\enVert{C}_1 := \abs{C_1} + \cdots + \abs{C_n} = C_1 + \cdots + C_n.
		\end{equation*}
		It is useful to recall that
		\begin{equation*}
			\enVert{C}_\infty \leq \enVert{C}_1 \leq n\enVert{C}_\infty.
		\end{equation*}
		By Lemma \ref{lem:coercive}, there exist constants $\delta \in (0,1)$ and $M > 0$ such that if $C_j = \enVert{C}_\infty \geq M$, then $F_j(C) \geq \delta C_j$.
		In this case,
		\begin{equation*}
			L_j(C) \leq \del{(1-\delta)C_j + a_j}_+ = (1-\delta)C_j + a_j.
		\end{equation*}
		In general,
		\begin{equation*}
			L_i(C) \leq C_i + a_i.
		\end{equation*}
		Summing over all $i$, we deduce that
		\begin{equation*}
			\enVert{L(C)}_1 \leq \enVert{C}_1 + \enVert{a}_1 - \delta C_j.
		\end{equation*}
		Recall that $C_j = \enVert{C}_\infty \geq n^{-1}\enVert{C}_1$.
		Let $r > \enVert{a}_1 n/\delta$.
		Then we have
		\begin{equation} \label{eq:L1C}
			\enVert{L(C)}_1 \leq \enVert{C}_1 + \enVert{a}_1 - \delta n^{-1}\enVert{C}_1 < \enVert{C}_1 \quad \text{if}~\enVert{C}_1 \geq r.
		\end{equation}
	
		We now restrict to a compact, convex set.
		Let $M_r$ be the set of all $x \in \mathbb{R}_{\geq0}^n$ such that $\enVert{x}_1 \leq r$. Then $M_r$ is closed and bounded, so $M_r$ is compact, and it is clearly nonempty. Lastly, $M_r$ is convex, as it is the intersection of two convex sets, namely, $\bb{R}^n_{\geq 0}$ and the closed unit ball for the 1-norm. Thus, by Lemma \ref{Brouwer}, any continuous function $g:M_r\rightarrow M_r$ has a fixed point.
	
		Define a function
		\begin{equation}
			P(D) := \begin{cases}
				\displaystyle D, \norm{D}_1 \leq r \\
				\frac{rD}{\norm{D}_1}, \norm{D}_1 > r.
			\end{cases}
		\end{equation}
		Note that $P$ maps $\bb{R}^n_{\geq 0}$ into $M_r$.
		Finally, let $Q(C) := P(L(C))$ for $C \in M_r$. Then $Q$ is continuous, as it is the composition of continuous functions, and it maps $M_r$ to itself. Thus, by Lemma \ref{Brouwer}, $Q$ has a fixed point $C \in M_r$. If $\norm{C}_1=r$, then 
		\begin{equation}
			r = \norm{C}_1 = \norm{Q(C)}_1 = \norm{P(L(C))}_1.
		\end{equation}
		By construction of $P$, this implies
		\begin{equation}
			\norm{L(C)}_1 \geq r.
		\end{equation}
		But this contradicts \eqref{eq:L1C}.
		It follows that $\norm{C}_1<r$. Thus, 
		\begin{equation}
			r> \norm{C}_1=\norm{Q(C)}_1 = \norm{P(L(C))}_1.
		\end{equation}
		Hence, $P(L(C))=L(C)$, so finally we have
		\begin{equation}
			C = Q(C) = P(L(C)) = L(C).
		\end{equation}
		This implies that for each $i$,
		\begin{equation*}
			C_i = \del{C_i + a_i - F_i(C)}_+.
		\end{equation*}
		If $C_i > 0$, this implies
		\begin{equation*}
			C_i = C_i + a_i - F_i(C) \Rightarrow F_i(C) = a_i.
		\end{equation*}
		If $C_i = 0$, then $F_i(C) = 0$ and we get
		\begin{equation*}
			0 = \del{a_i}_+ = a_i
		\end{equation*}
		which contradicts the assumption $a_i > 0$.
		Hence $C_i > 0$ for all $i$ and $F(C) = a$, as desired.
	\end{proof}

	\section{Proving Uniqueness of a Solution} \label{sec:uniqueness}
	
	By Theorem \ref{MAIN}, we know that for any vector $a$ with positive entries there exists a solution $C \in \bb{G}$ to $F(C)=a$.	
	In this section, we prove that the solution is unique. To do so, we show that $F(C)$ is strictly monotone in a classical sense.
	It is interesting to note that monotonicity is often used to prove uniqueness in mean field games, the most commonly used notion being the ``Lasry-Lions monotonicity condition'' (cf.~\cite{lasry07}).
	(For a discussion on alternative monotonicity conditions, see \cite{graber2023monotonicity}.)
	Here, the coupling is indeed precisely through the density of players, and therefore the Lasry-Lions monotonicity condition holds.
	However, it is interesting to note that the monotonicity of $F(C)$, a function of the parameters $C_1,\ldots,C_n$, is not a direct consequence the Lasry-Lions monotonicity condition in any obvious way.
	
	\begin{lemma}
		On $\mathbb{G}$, $F(C)$ is strictly monotone, i.e. for any $C,\tilde{C} \in \mathbb{G}, (F(C)-F(\tilde{C})(C-\tilde{C})>0$.\label{MONOTONE}
	\end{lemma}
	\begin{proof}
		Recall that $\mathbb{G}$ is convex by Corollary \ref{cor:length of Ej}.
		To show F is strictly monotone, then, it suffices to show that
		\begin{equation*}
			\sum_{i=1}^{n}\sum_{j=1}^{n}\dpd{F_i(C)}{C_j} \cdot(v_iv_j)>0 \ \forall \ {C},{v} \in\mathbb{R}^n \text{ with } v\neq0.
		\end{equation*}
		Assume for the sake of notation that $v_0=0=v_{n+1}$. By Lemma 2.11,
		\begin{equation*}
			S:= \sum_{i=1}^{n}\sum_{j=1}^{n}\dpd{F_i(C)}{C_j}v_iv_j = 
			\displaystyle\sum_{j=1}^{n}\del{\dpd{F_j(C)}{C_j}(v_{j})^2+\dpd{F_j(C)}{C_{j-1}}(v_{{j-1}}v_{j})+\dpd{F_j(C)}{C_{j+1}}(v_{j}v_{{j+1}})}.
		\end{equation*}
		Using the derivative formulas from \eqref{DFj},
		\begin{multline*}
			S = \sum_{j=1}^{n}\del{\norm{E_{j}}+\frac{f_{j}(\alpha_{j})}{2(x_{j}-x_{j-1})}+\frac{f_{j}(\beta_{j})}{2(x_{j+1}-x_{j})}}\del{v_{j}}^2\\
			-\del{\frac{f_{j}(\alpha_{j})}{2(x_{j}-x_{j-1})}}\del{v_{j-1}v_{j}}-\del{\frac{f_{j}(\beta_{j})}{2(x_{j+1}-x_j)}}\del{v_jv_{j+1}}.
		\end{multline*}
		Since $f_1(\alpha_1)=0=f_n(\beta_n)$, we can rewrite this as
		\begin{multline*}
			S = \sum_{j=1}^{n}{\norm{E_{j}}}{v_j}^2+\sum_{j=2}^{n}\del{\frac{f_{j}(\alpha_{j})}{2(x_{j}-x_{j-1})}\del{v_j}^2-\del{\frac{f_{j}(\alpha_{j})}{2(x_{j}-x_{j-1})}}\del{v_{j-1}v_{j}}}\\
			+\sum_{j=1}^{n-1}\del{\frac{f_{j}(\beta_{j})}{2(x_{j+1}-x_{j})}\del{v_{j}}^2-\del{\frac{f_{j}(\beta_{j})}{{2(x_{(j+1)}-x_j)}}}\del{v_jv_{j+1}}}.
		\end{multline*}
		Since $f_{j+1}(\alpha_{j+1})$ is either $f_j(\beta_j)$ or $0$, we can combine to get
		\begin{equation*}
			\begin{split}
				S
				=\displaystyle&\sum_{j=1}^{n}{\norm{E_{j}}}{v_j}^2 +\sum_{j=1}^{n-1}\frac{f_{j}(\beta_{j})}{2(x_{j+1}-x_{j})}\del{v_{j+1}}^2-\frac{2f_{j}(\beta_{j})}{2(x_{j+1}-x_{j})}\del{v_{j}v_{j+1}}+\frac{f_{j}(\beta_{j})}{2(x_{j+1}-x_{j})}\del{v_{j}}^2\\
				=\displaystyle&\sum_{j=1}^{n}{\norm{E_{j}}}{v_j}^2 +\sum_{j=1}^{n-1}\del{\frac{f_{j}(\beta_{j})}{2(x_{j+1}-x_{j})}}\del{v_{j+1}-v_j}^2 >0.
			\end{split}
		\end{equation*}
		Thus, $F(C)$ is strictly monotone on $\mathbb{G}$. 
	\end{proof}
	
	We now have everything we need to prove uniqueness.
	\begin{theorem} \label{thm:uniqueness}
		The solution to the equation $F(C)=a$ is unique.
	\end{theorem}
	\begin{proof}
		Suppose, for the sake of contradiction, that $C_1,C_2 \in \mathbb{G}$ with $C_1 \neq C_2$ are both solutions to $F(C)=a$. Then $F(C_1) = F(C_2)$, so $(F(C_1)-F(C_2)(C_1-C_2)=0$. But this contradicts \ref{MONOTONE}. Thus, the solution to $F(C)$ is unique.
	\end{proof}
	
	\begin{remark} \label{rem:convexity}
		As a result of Lemma \ref{MONOTONE}, the potential $\Phi$ defined in Remark \ref{rem:potential} is strictly convex on $\bb{G}$.
		The unique solution of $F(C) = a$ can be characterized as the unique minimizer of $\Phi(C) - C \cdot a$ over $C \in \bb{G}$.
	\end{remark}
	We have now shown that a unique solution to $F(C)=a$ exists and is unique. We now move on to numerical approximations.
	
	\section{Numerical Simulations}
	
	We used Newton's Method to find approximate solutions to $F(C) = a$.
	By algebraically solving $\frac{4}{3}C_j^{3/2} = a_j$, we obtain an initial guess $C = C^{(0)}$ such that $F_j(C) \leq a_j$ for each $j$ (see Remark \ref{rem:max Fj}).
	Then we define an approximating sequence $C^{(n)}$ by
	\begin{equation}
		C^{(n+1)}=C^{(n)}-DF\del{C^{(n)}}^{-1}\del{F\del{C^{(n)}} - a},
	\end{equation}
	where $DF(C)$ is the Jacobian matrix computed in Section \ref{sec:formulas}.
	The residual $\enVert{F\del{C^{(n)}} - a}$ can be computed easily; in practice it becomes vanishingly small after only a few iterations.
	Below we report on three simulations.
	In each case, the initial distribution $m = \sum_{j=1}^n a_j \delta_{x_j}$ is illustrated a set of $n$ bars located at position $x_j$ with height $a_j$.
	Then the final density $f(x)$ is then graphed over these bars.
	The function $f(x)$ is piecewise quadratic with the $j$th parabola centered at $x_j$ and representing the spreading of the population that starts at $x_j$; the area under this parabola above the $x$-axis equals $a_j$.
	This accords with our geometric intuition about the game: players spread out as much as possible so as to avoid areas of high population density, and the shape of the final density matches the quadratic cost $(x-y)^2$ of traveling from $x$ to $y$.
	
	\begin{figure}[H]
		\centering
		\includegraphics[width = \textwidth]{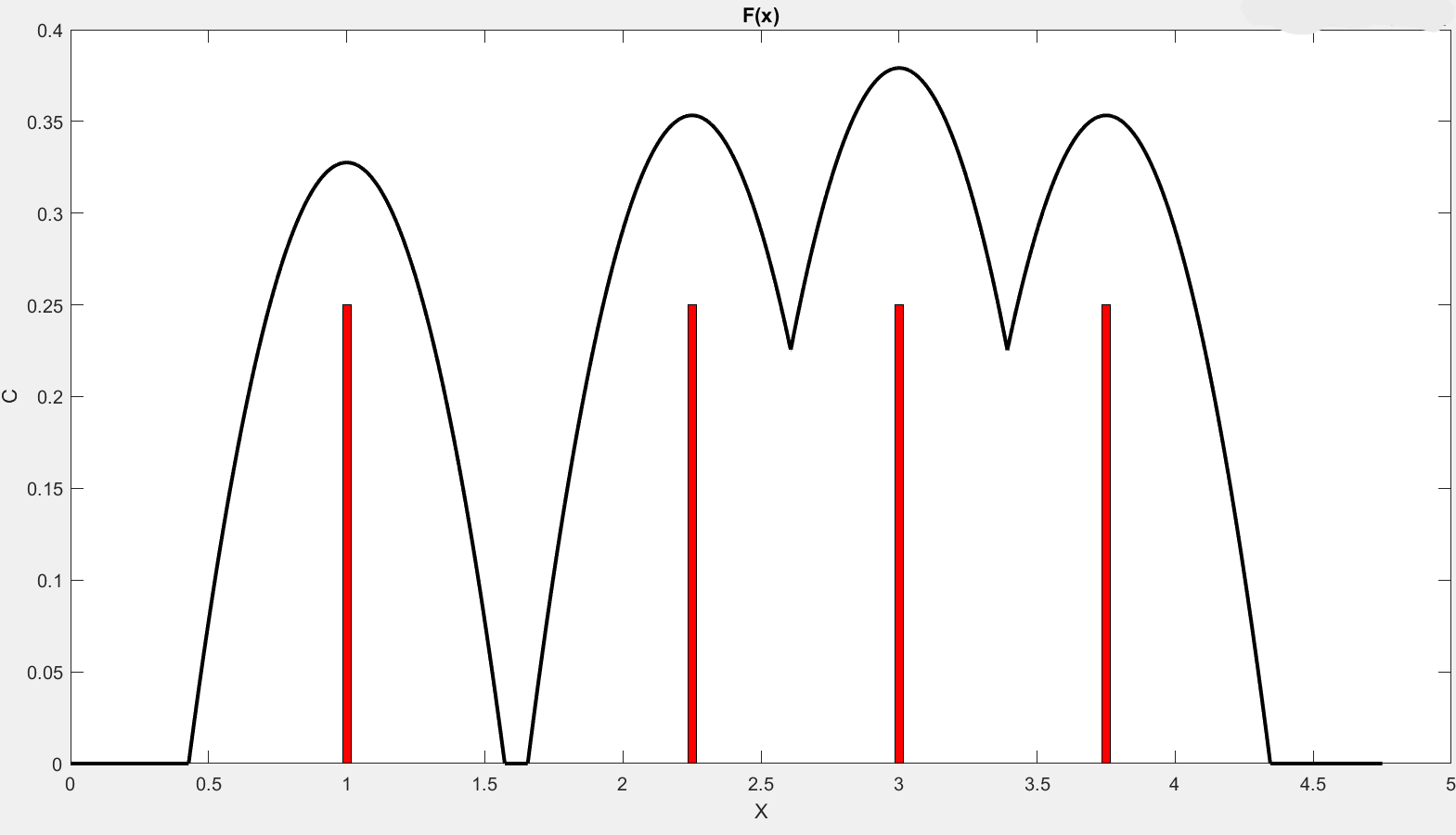}
		\caption{Example 1: $x=(1,2.25,3,3.75), a=(0.25,0.25,0.25,0.25)$}
		\label{fig:example1}
	\end{figure}
	In the first simulation (Figure \ref{fig:example1}), we took an initial population uniformly distributed over four points: $m = \frac{1}{4}\del{\delta_1 + \delta_{2.25} + \delta_3 + \delta_{3.75}}$.
	In this example, the leftmost quarter of the population is far enough away from rest that it can spread out without colliding with the others, whereas the other three are close enough to each other that players have less room to move, resulting in a narrower probability density.
	
	\clearpage
	\pagebreak
	\begin{figure}
		\centering
		\includegraphics[width = \textwidth]{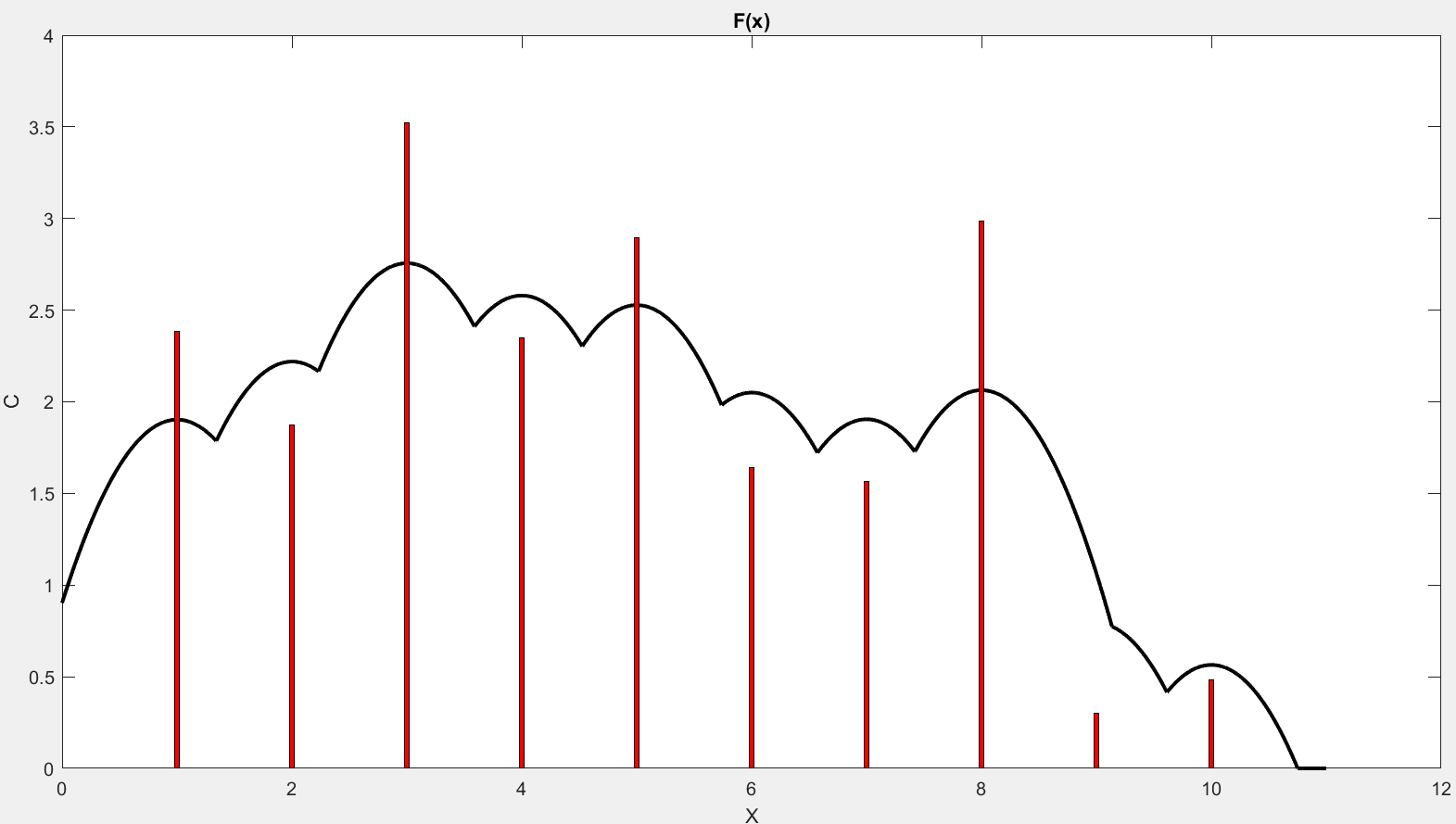}
		\caption{Example 2: $x=(1,2,\dots,10), a=random$}
		\label{fig:example2}
	\end{figure}
	
	\begin{figure}
		\centering
		\includegraphics[width = \textwidth]{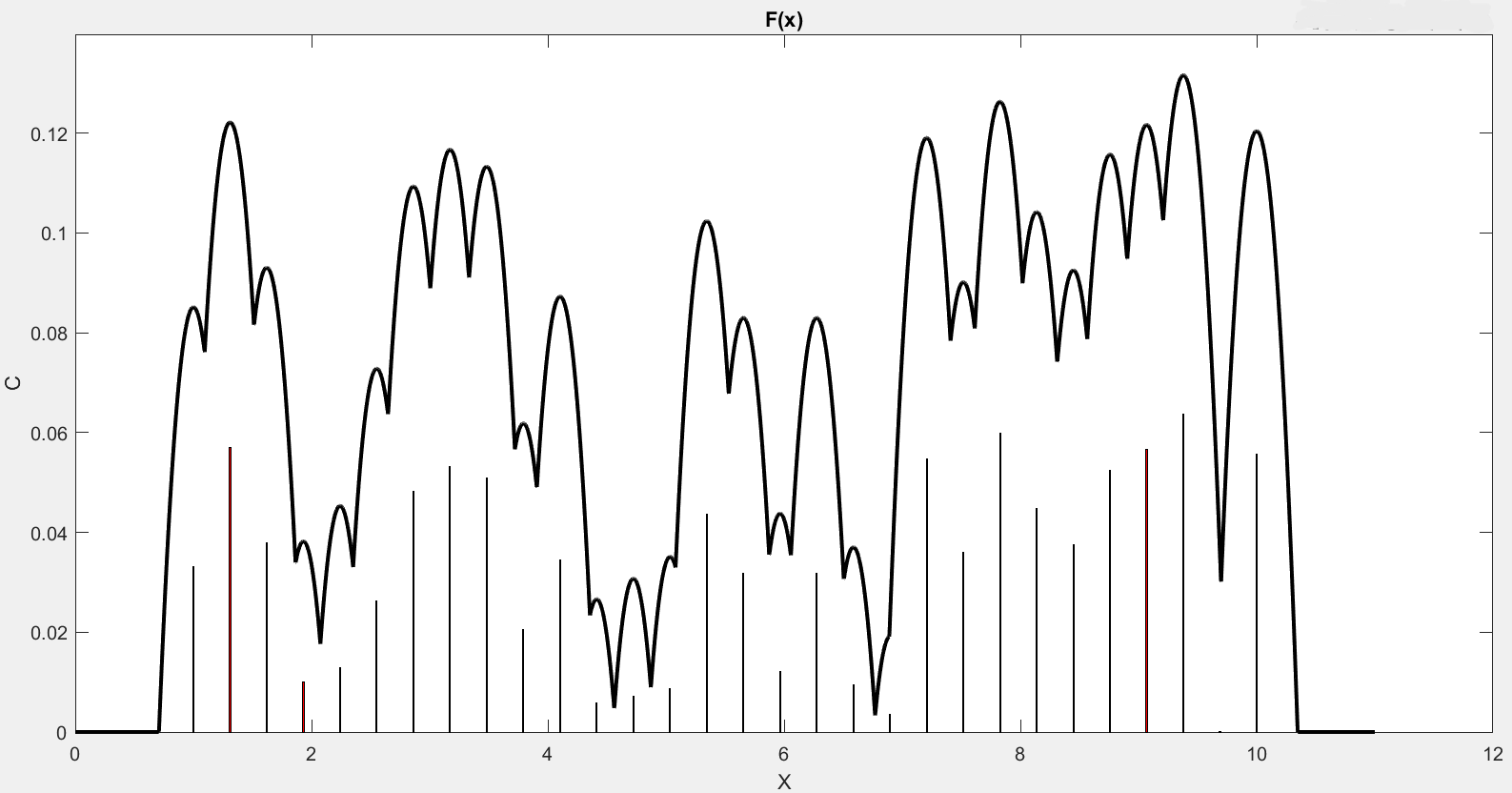}
		\caption{Example 3: $x=random$, $a=random$}
		\label{fig:example3}
	\end{figure}
	\clearpage
	\pagebreak

	\newpage
	In the second simulation (Figure \ref{fig:example2}), we let $x_j - x_{j-1}$ be constant and let $a$ be random. (Note that we did not normalize the components of $a$ to sum to 1, but this does not change the shape of the solution.) Notice how in this sample the probability density has a very interesting contour. For example, the ``bubble" (i.e.~piece of parabola) corresponding to $x=9$ barely appears because it is crowded by neighboring ``bubbles.'' This demonstrates how the different weights locally change the shape of the density.

	In the last simulation (Figure \ref{fig:example3}), the vectors $x$ and $a$ are random (and $a$ is normalized to be a probability vector).
	Although the initial measure is highly irregular, the final density will still be smooth, though it appears less so as the number of random points are chosen.
	Indeed, as the proof of uniqueness in Section \ref{sec:uniqueness} suggests,  the problem becomes more and more ill-conditioned as the points $x_j$ get closer together, since the lengths of the intervals $E_j$ necessarily get smaller.
	
	As a final remark, we recall that $F = \nabla \Phi$ for some potential $\Phi$, and that solving $F(C) = a$ is equivalent to minimizing $\Phi(C) - C \cdot a$ (see Remarks \ref{rem:potential} and \ref{rem:convexity}).
	Since $\Phi$ is strictly convex, one might hope that Newton's method converges globally.
	However, it is not true in general that strict convexity implies global convergence for Newton's method, as explained in this Stack Exchange discussion \cite{newton-rhapson-stack}; see also \cite{boyd2004convex}.
	The proof of Lemma \ref{MONOTONE} suggests that the lower bound on the Hessian $\nabla^2 \Phi(C)$ goes to zero as any of the lengths $\enVert{E_j}$ vanishes.
	One can infer that the problem is better conditioned if the points $x_j$ are sufficiently spread out and/or the weights $a_j$ are such that the equilibrium density has a uniform lower bound on the lengths of the intervals $E_j$.
	We see this in numerical experiments: if the distance between points $x_j$ becomes too small, the algorithm appears not to converge.

	\section{Conclusion}
	
	In this paper we have studied a mean field game with final cost equal to the population density of players.
	We considered the case where the initial measure is discrete, which seems to be new in the literature.
	We showed that the problem reduces to a finite dimensional problem, which can prove is well-posed by classical methods.
	Our numerical simulations illustrate the ``smoothing effect'' that mean field games with density penalization are expected to have.
	For future research, it would be interesting to see if this approach could be extended to more general mean field games with density-dependent costs and discrete initial measures.
	
	\bibliographystyle{alpha}
	\bibliography{../../../mybib/mybib}
\end{document}